\documentclass[leqno,12pt,a4paper]{amsart}
\usepackage{amsmath}
\usepackage{amsfonts}
\usepackage{amssymb}
\usepackage{amsthm} %proof
\usepackage{mathtools}\mathtoolsset{showonlyrefs=true}
\usepackage{bm}
\usepackage{float}
\usepackage{stackengine}
\allowdisplaybreaks[1]

\usepackage{color}
\usepackage{pdfsync}
\usepackage[setpagesize=false]{hyperref}
\usepackage[hdivide={2.5cm,,2.5cm}, vdivide={3.5cm,,2.8cm}]{geometry}

\title[Continuous Evolution Families]{Continuous Evolution Families}
\author[S. Hoshinaga]{Shota Hoshinaga}
\author[I. Hotta]{Ikkei Hotta}
\author[H. Yanagihara]{Hiroshi Yanagihara}
\thanks{The second author was supported by JSPS Grant-in-Aid for Scientific Research(C) 20K03632}
\thanks{The third author was supported by JSPS Grant-in-Aid for Scientific Research(C) 19K03519}
\address{Department of Applied Science, Yamaguchi University 2-16-1 Tokiwadai, Ube 755-8611, Japan}
\email{hoshinaga0727@gmail.com}
\email{ihotta@yamaguchi-u.ac.jp}
\email{hiroshi@yamaguchi-u.ac.jp}
\date{\today}

\newtheorem{theorem}{Theorem}
\newtheorem{lemma}[theorem]{Lemma}

\newtheorem{problem}[theorem]{Problem}

\theoremstyle{definition}
\newtheorem{definition}[theorem]{Definition}
\newtheorem{remark}[theorem]{Remark}

\newcommand{\C}{\mathbb{C}}%complex plane
\newcommand{\D}{\mathbb{D}}%unit disk
\newcommand{\N}{\mathbb{N}}%natural number
\newcommand{\R}{\mathbb{R}}%real number
\newcommand{\B}{\mathfrak{B}}%Bazilevic functions   
\numberwithin{equation}{section}
\numberwithin{theorem}{section}

\DeclareMathOperator{\id}{id}

\def\2sidelim{%
\lim_{{\substack{\scriptscriptstyle s \leq c \leq t \\%
\scriptscriptstyle  t - s \searrow 0} }}}

\def\qed{\hfill $\Box$}

\begin{document}

\begin{abstract}
Recently in relation to the theory of non-commutative probability, a notion of evolution families $\{\omega_{s,t}\}_{s \le t}$ is generalized that are only continuous in parameters, namely $(s,t) \mapsto \omega_{s,t}$ is continuous with respect to locally uniform convergence on a planar domain.
In this article we present various equivalence conditions to the continuous evolution families concerned with the left and right parameters.
We also provide an example of a discontinuous evolution family in the last section.
\end{abstract}

\maketitle

\vspace{-20pt}

\tableofcontents

\section{Introduction}

\subsection{Evolution families}
A theory of one/two-parameter families of holomorphic functions is seen as a key ingredient in various areas of analysis, e.g. geometric function theory, operator theory, probability theory and so on. In this article we focus on the so-called evolution families, which play an essential role in the theory of the Loewner equations. 
It is usually assumed to be absolutely continuous (which implies a.e. differentiability) in time, see for example \cite{BracciCD:evolutionI}. 
%Then there is one-to-one correspondence between such a family of functions and the differential equation. 
Then such a family of functions can be characterized by the ordinary differential equation.

Recently, in connection with non-commutative probability, the following general concept of evolution families was introduced.

\begin{definition}[{\cite[Definition 3.1]{MR4152669}}]
\label{reverse-evolution}
A two-parameter family $\{f_{s,t}\}_{0 \le s \le t < \infty}$ of holomorphic self-mappings $f_{s,t} : D \to D$ on some planar domain $D$ is said to be a \textit{reverse evolution family}  if the following conditions are satisfied;
\begin{enumerate}%[label=\rm(TM\arabic*)]
\item[(TM1)] \label{TM1} $f_{s,s}(z) =z$ for all $z \in D$ and $s \ge 0$,
\item[(TM2)] \label{TM2} $f_{s,u} \circ f_{u,t} = f_{s,t}$ for all $0 \le s \le u \le t < \infty$,
\item[(TM3)] \label{TM3} $(s,t)\mapsto f_{s,t}$ is continuous with respect to the topology of locally uniform convergence on $D$.
\end{enumerate} 
%and each mapping $f_{s,t}$ is called a \textit{transition mapping}. 
\end{definition}

In the above definition, $\{f_{s,t}\}$ needs not to be differentiable in parameters, and hence one cannot expect to obtain a usual differential equation it satisfies. 
Some researches contribute to this difficulty, see \cite{HasebeHotta2021} and \cite{Yanagihara:2022}.
%There are some previous researches for this difficulty, see \cite{HasebeHotta2021} and \cite{Yanagihara:2022}.

In the study of the above general evolution families, we have come to ask ourselves the following question; \textit{can the condition (TM3) be replaced to more weak one, for example locally uniform continuity in either parameter s or t?}
In this article we give some answer to this problem.

\subsection{Notations and definitions}
Let $\mathbb{C}$ be the complex plane and 
$\mathcal{A}$ be the space of all analytic functions $f$ 
in the unit disk $\mathbb{D} := \{ z \in \mathbb{C} : |z| < 1 \}$
endowed with the topology of locally uniformly convergence on $\mathbb{D}$.
Let 
$\B 
:= \{ f \in \mathcal{A} : f( \mathbb{D}) \subset \mathbb{D} \} 
\subset \mathcal{A}$.
We denote the identity mapping on $\mathbb{D}$ by $\id_{\mathbb{D}}$.

%For 
%$z_0, z_1 \in \mathbb{D}$,
%the \textit{pseudo-hyperpolic distance 
%$d^*(z_0,z_1)$ between 
%$z$ and $w$} is  defined by
%\begin{align}
% d^*(z_0,z_1) := & \, \frac{|z_0-z_1|}{|1-\overline{z_1}z_0|}.
%%\\
%%  d(z_0,z_1) = & \, 
%% \log \frac{1+\frac{|z_0-z_1|}{|1-\overline{z_1}z_0|} }
%% {1-\frac{|z_0-z_1|}{|1-\overline{z_1}z_0|} } .
%\end{align}
%The Schwarz-Pick lemma states
%that each $\omega \in \B$ is distance decreasing 
%with respect to the pseudo-hyperbolic distance, i.e.,
%\begin{equation}
%   d^*(\omega(z_0),\omega(z_1)) \leq d^*(z_0,z_1) \quad 
%   \text{ for all } z_0, z_1 \in \mathbb{D} .
%\end{equation}

Next, we give a definition of evolution families in a naive form.
Here $I \subset [- \infty, \infty  ]$ be an interval and
$I_+^2 := \{ (s,t) : s,t \in I \textrm{ with } s \leq t \}$.
\begin{definition}
\label{def:evolutiuon-family}
A family $\{ \omega_{s,t} \}_{(s,t) \in I_+^2}$ in $\B$
is said to be
an \textit{evolution family} if it has the following three properties;
\begin{itemize}
 \item[(EF1)]
 $\omega_{s,t} $ is non-constant for all $(s,t) \in I_+^2$,
 \item[(EF2)] 
$\omega_{t,t} = \id_\mathbb{D}$ for all $t \in I$,
 \item[(EF3)]
$\omega_{u,t} \circ \omega_{s,u} = \omega_{s,t}$
for $s,u,t \in I$ with $s \leq u \leq t$.
\end{itemize} 
\end{definition}

Let $\Lambda$ be a metric space and $\lambda_0 \in \Lambda$.
A family $\{ f_\lambda \}_{\lambda \in \Lambda}$ 
in $\mathcal{A}$ is said to be \textit{continuous at $\lambda_0$} if
the mapping $\Lambda \ni \lambda \mapsto f_\lambda \in \mathcal{A}$ 
is continuous at $\lambda_0$,
i.e., 
$f_\lambda$ 
converges locally uniformly to $f_{\lambda_0}$ on $\mathbb{D}$
as $\lambda \rightarrow \lambda_0$.
When $\{ f_\lambda \}_{\lambda \in \Lambda}$ is continuous at 
all $\lambda_0 \in \Lambda$, we simply say
\textit{$\{ f_\lambda \}_{\lambda \in \Lambda}$ is continuous}.
It is easy to see that 
if $\{ f_\lambda \}_{\lambda \in \Lambda}$ and $\{ g_\lambda \}_{\lambda \in \Lambda}$ 
are continuous, then 
$\{ \alpha f_\lambda + \beta g_\lambda \}_{\lambda \in \Lambda}$ 
is continuous for each $\alpha, \beta \in \mathbb{C}$.
Furthermore
when $\{ g_\lambda \}_{\lambda \in \Lambda} \subset \B$,
 $\{ f_\lambda \circ g_\lambda \}_{\lambda \in \Lambda}$ 
is also continuous.

Now we add some continuity to the above evolution families as follows.
\begin{definition}
An evolution family $\{ \omega_{s,t} \}_{(s,t) \in I_+^2}$ is said to be \textit{jointly continuous} if;
\begin{itemize}
 \item[(EF4)] the mapping $I_+^2 \ni (s,t) \mapsto \omega_{s,t} \in \mathcal{A}$
is continuous.
\end{itemize}
\end{definition}

\begin{remark}
It is known in \cite[Lemma 3.5]{MR4152669} that (EF1) follows from (EF2)-(EF4).
In other words, (EF1) is superfluous to define jointly continuous evolution families (compare with \cite[Definition 3.1]{MR4152669}). 
\end{remark}

\begin{remark}
A simple parameter change gives the local duality between evolution families and reverse evolution families.
%Let $-I := \{x : -x \in I\}$. 
Precisely, one can show that $\{ \omega_{s,t} \}_{(s,t) \in I_+^2}$ is a jointly continuous evolution family if and only if $\{ f_{s,t} := \omega_{-t, -s} \}_{(s,t) \in -I_+^2}$ is a reverse evolution family in the sense of Definition \ref{reverse-evolution}, where $-I := \{x : -x \in I\}$  (see also \cite{contreraslocalduality}).
Hence our results for jointly continuous evolution families immediately can be applied to reverse evolution families.
\end{remark}

We will introduce some more terminologies to state our main results.
In the case that $I=[a,b]$ with $-\infty < a < b < \infty$,
we say $\{ \omega_{s,t} \}_{(s,t) \in I_+^2}$ is \textit{continuous 
with respect to the right} (or \textit{left}) \textit{parameter} if
the mapping $I \ni t \mapsto \omega_{a,t} \in \mathcal{A}$ 
(or $I \ni s \mapsto \omega_{s,b} \in \mathcal{A}$, respectively) 
is  continuous.
%This We employ this weak assumption than the usual right continuity
This is a somewhat weaker condition than the usual continuity in the right (or left) parameter, i.e., \textit{for all fixed $c \in [a,b]$,  $[c,b] \ni t \mapsto \omega_{c,t} \in \mathcal{A}$ is  continuous}.
As a mater of fact, the above usual continuity follows from our weak continuity by virtue of the semigroup property of the evolution family (EF3). % we have employed.

We say that an evolution family $\{ \omega_{s,t} \}_{(s,t) \in I_+^2}$ 
is \textit{hyperbolically bounded} if 
\begin{equation}
\label{equationhyp_bdd}
  \sup_{(s,t) \in I_+^2} |\omega_{s,t}(0) | < 1 
\end{equation}
and \textit{locally hyperbolically bounded} if 
$\{ \omega_{s,t} \}_{(s,t) \in I_+^2}$ 
is hyperbolically bounded on any compact subinterval of $I$.
Notice that in \eqref{equationhyp_bdd}, $\omega_{s,t}(0)$ can be replaced to $\omega_{s,t}(z_{0})$ for any fixed $z_0 \in \mathbb{D}$ to obtain equivalent conditions.
This easily follows from Lemma \ref{lemma:basic-estimatesI}
which will be stated in the next section.
It is also evident that a jointly continuous evolution family is locally hyperbolically bounded.
%In fact, if $\{ \omega_{s,t} \}_{(s,t) \in I_+^2}$ is a jointly continuous evolution family and $J$ is a compact subinterval of $I$, then $\omega_{s,t}(0)$ is hyperbolically bounded on $J_+^2$.
%Let $\{ \omega_{s,t} \}_{(s,t) \in I_+^2}$ be a jointly continuous evolution family.
%If $J$ is a compact subinterval of $I$,
%then $\omega_{s,t}(0)$ is hyperbolically bounded on
%$J_+^2$.
%Therefore jointly continuous evolution families are locally hyperbolically bounded.

\subsection{Main results}

We are ready to present the main results.
In brief, jointly continuity follows from continuity with respect to the right parameter, but not immediately from the left parameter.

Our first theorem is about the right parameter of an evolution family.
\begin{theorem}
\label{thm:right_parameter} 
Let $\{ \omega_{s,t} \}_{(s,t) \in I_+^2}$ be an evolution family
on $I=[a,b]$ with $-\infty < a < b< \infty$. 
Then the following are equivalent;
\begin{itemize}
 \item[{\rm (i)}]
The family $\{ \omega_{s,t} \}_{(s,t) \in I_+^2}$ is jointly continuous,
 \item[{\rm (ii)}]
The family $\{ \omega_{s,t} \}_{(s,t) \in I_+^2}$
is continuous with respect to the right parameter,
%i.e., $\{ \omega_{a,t} \}_{t \in I}$ is continuous in $t \in [a,b]$.
 \item[{\rm (iii)}]
For each fixed $z_0 \in \mathbb{D}$, 
$\omega_{a,t}(z_0)$ and $\omega_{a,t}'(z_0)$ are continuous functions  
of $t \in [a,b]$,
 \item[{\rm (iv)}]
For some $z_0 \in \mathbb{D}$, 
$\omega_{a,t}(z_0)$ and $\omega_{a,t}'(z_0)$ are continuous functions 
of $t \in [a,b]$.
\end{itemize}
Furthermore in these cases 
$\omega_{s,t}$ is univalent on $\mathbb{D}$ for each $(s,t) \in I_+^2$.
\end{theorem}

\begin{remark}
For univalence of $\omega_{s,t}$, see also \cite[Theorem 3.16]{MR4152669}.
\end{remark}

Next, we show some equivalence conditions concerned with the left parameter of an evolution family.
%for an evolution family is  jointly continuous.

\begin{theorem}
\label{thm:continuity-wrt-s}
Let $\{ \omega_{s,t} \}_{(s,t) \in I_+^2}$ be an evolution family
on $I=[a,b]$ with $-\infty < a < b< \infty$. 
Then the following are equivalent;
\begin{itemize}
 \item[{\rm (i)}]
The family $\{ \omega_{s,t} \}_{(s,t) \in I_+^2}$ is jointly continuous, 
 \item[{\rm (v)}]
The family $\{ \omega_{s,t} \}_{(s,t) \in I_+^2}$ is hyperbolically bounded and continuous with respect to the left parameter, 
% \item[{\rm (vi)}]
%The family $\{ \omega_{s,t} \}_{(s,t) \in I_+^2}$ is hyperbolically bounded and $\omega_{s,b}$ is univalent on $\D$ for all $s \in [a,b]$,
  \item[{\rm (vi)}]
The family $\{ \omega_{s,t} \}_{(s,t) \in I_+^2}$ is hyperbolically bounded,
$\2sidelim \omega_{s,t}(0) = 0$ for all $c \in [a,b]$
and $\omega_{s,b}'(0)$ is a continuous function of $s \in [a,b]$
with $\omega_{s,b}'(0) \not= 0$ for $s \in [a,b]$. 
\end{itemize}
\end{theorem}

In order to show Theorem \ref{thm:continuity-wrt-s}, we need the following further equivalences.
%Further, some more equivalence to jointly continuity are given. 
%Further, some more equivalence are obtained.
Here $D(I)\subset I_+^2$ is a diagonal set of $I_+^2$, namely $D(I) := \{ (t,t) : t \in I \}$.
%Next $\{ \omega_{s,t} \}_{(s,t) \in I_+^2}$ is jointly 
%continuous on $I_+^2$
%if and only if it is jointly continuous on the diagonal set $D(I)$.

\begin{theorem}
\label{thm:cont_on_diagonal_set}
Let $\{ \omega_{s,t} \}_{(s,t) \in I_+^2}$ 
be an evolution family on $I= [a,b]$ with $-\infty < a < b < \infty$. 
Then the following are equivalent;
\begin{itemize}
 \item[{\rm (i)}]
The family $\{ \omega_{s,t} \}_{(s,t) \in I_+^2}$ is jointly continuous,
 \item[{\rm (vii)}]
The family $\{ \omega_{s,t} \}_{(s,t) \in I_+^2}$ is jointly continuous
on $D(I)$,
 \item[{\rm (viii)}]
The family $\{ \omega_{s,t} \}_{(s,t) \in I_+^2}$ is hyperbolically bounded with
$\2sidelim \omega_{s,t}(0) = 0$ and $\2sidelim \omega_{s,t}'(0) = 1$
for all $c \in [a,b]$.
\end{itemize}
\end{theorem}

This paper is structured as follows.
In Section \ref{Section2}, we collect preliminary results that are used throughout the later discussions. 
In Section \ref{Section3}, we prove our main results: Theorem \ref{thm:right_parameter}, Theorem \ref{thm:continuity-wrt-s} and Theorem \ref{thm:cont_on_diagonal_set}.
In Section \ref{Section4} we will close this paper with some remarks on our results.

%We conclude Section \ref{Section4} and this paper with a brief consideration of the Loewner Range $\bigcup_{t \ge 0} f_{t}(\D)$.

\section{Preliminaries}
\label{Section2}

\begin{lemma}
\label{lemma:basic-estimatesI}
Let $\omega \in \B$. 
Then we have
\begin{equation}
\label{Lem2-1first}   
|\omega (z)| \leq  \,  \frac{|z| + |\omega(0)| }{1+| \omega(0)| |z|}
\end{equation}
and
\begin{equation}
\label{Lem2-1second}   
| \omega (0)| 
  \leq  \, \frac{|z|+|\omega(z)|}{1+|z||\omega(z)|} . 
\end{equation}
for all $z \in \mathbb{D}$.
\end{lemma}

\begin{proof}
Let $z_{0} \in \D$.
By the Schwarz-Pick lemma we have
\begin{equation}
\left| \frac{\omega(z) - \omega(z_{0})}{1 - \overline{\omega(z_{0})} \omega(z)} \right|
\leq \delta :=
\left| \frac{z - z_{0}}{1 -  \overline{z_{0}} z} \right|,
\end{equation}
which is equivalent to 
\begin{equation}
\label{hyp-to-eucl-disk}
\left| \omega(z) - \frac{(1-\delta^{2})\omega(z_{0})}{1- \delta^{2} |\omega(z_{0})|^2} \right|
\leq \frac{(1-|\omega(z_{0})|^2)\delta}{1- \delta^{2}|\omega(z_{0})|^2}.
\end{equation} 
Then
\begin{equation}
|\omega(z)| \le \frac{(1-\delta^{2})\omega(z_{0})}{1- \delta^{2} |\omega(z_{0})|^2} + \frac{(1-|\omega(z_{0})|^2)\delta}{1- \delta^{2}|\omega(z_{0})|^2}
=
\frac{\delta + |\omega(z_{0})| }{1+|\omega(z_{0})|\delta}.
\end{equation}
From this \eqref{Lem2-1first} and \eqref{Lem2-1second} follow as $z_{0}=0$ and $z = 0$, respectively.
\end{proof}

%
%
%\
%
%\
%
%\
%
%
%
%Let $\alpha := \omega(0) \in \mathbb{D}$.
%By the Schwarz-Pick lemma we have
%\[
%   \left| \frac{\omega(z) - \alpha }{1- \overline{\alpha} \omega(z)} \right|
% \leq
%   | z| ,
%\]
%which is equivalent to 
%\begin{equation}
%\label{hyp-to-eucl-disk}
%   \left| \omega(z) - \frac{(1-|z|^2)\alpha}{1- |\alpha|^2 |z|^2} \right|
% \leq \frac{(1-|\alpha|^2)|z|}{1- |\alpha|^2 |z|^2} .
%\end{equation}
%From this the first inequality \eqref{Lem2-1first} follows as 
%\[
% |\omega(z)|
% \leq
% \frac{(1-|z|^2)|\alpha|}{1- |\alpha|^2 |z|^2}
% +
% \frac{(1-|\alpha|^2)|z|}{1- |\alpha|^2 |z|^2} 
% =
% \frac{|z| + |\alpha| }{1+| \alpha | |z|}.
%\]
%
%Next, denote $w_{0} := \omega(z_0)$ for a fixed $z_{0} \in \D$. 
%Then we have by the Schwarz-Pick inequality
%\[
% \left| \frac{\omega(z)-w_0}{1-\overline{w_0} \omega(z)} \right|
% \leq
% \left| \frac{z-z_0}{1-\overline{z_0} z}  \right| .
%\]
%Substituting $z=0$ we obtain
%\[
% \left| \frac{\omega(0)-w_0}{1-\overline{w_0} \omega(0)} \right|
% \leq
% \left| z_0 \right|,
%\]
%which is equivalent to
%\[
%   \left| \omega(0) - \frac{(1-|z_0|^2)w_0}{1-|z_0|^2|w_0|^2}\right|
% \leq
% \frac{(1-|w_0|^2)|z_0|}{1-|z_0|^2|w_0|^2}.
%\]
%From this the second inequality \eqref{Lem2-1second} follows as
%\[
% |\omega(0)|
% \leq
% \frac{(1-|z_0|^2)|w_0|}{1-|z_0|^2|w_0|^2}
% +
% \frac{(1-|w_0|^2)|z_0|}{1-|z_0|^2|w_0|^2}
% =
% \frac{|z_0| + |w_0| }{1+|w_0 | |z_0|}.
%\]

\begin{lemma}
\label{lemma:basic-estimatesII}
Let $\omega \in \B$ 
with $\omega(0)=0$ and $\omega'(0) =: \lambda \in \overline{\mathbb{D}}$.
Then we have
\begin{equation}
\label{ineq:growth-bound-for-omega}
\left| 
 \omega (z)
 \right|
 \leq  
 |z| \frac{|z|+ |\lambda| }{1+ |\lambda ||z|}
 \end{equation}
 and
 \begin{equation}
\label{ineq:difference-between-omega-and-identity}
\left| 
 \omega(z) - z
 \right|
 \leq 
 \frac{|z|(1+|z|)|1 - \lambda|}{1- |\lambda ||z|} .
\end{equation}
for all $z \in \D$.
\end{lemma}
\begin{proof}
The first inequality \eqref{ineq:growth-bound-for-omega} immediately follows from \eqref{Lem2-1first} in Lemma \ref{lemma:basic-estimatesI} because the function $g(z) := \omega(z)/z$ belongs to $\B$.

Next, by \eqref{hyp-to-eucl-disk} $g \in \B$ satisfies
$$
\left| g(z) - \frac{(1-|z|^{2})\lambda}{1-|\lambda|^{2}|z|^{2}}\right| 
	\le \frac{(1-|\lambda|^{2})|z|}{1-|\lambda|^{2}|z|^{2}}.
$$
Then we have 
\begin{align}
|g(z)-1| 
	& \le \left| g(z) - \frac{(1-|z|^{2})\lambda}{1-|\lambda|^{2}|z|^{2}}\right| + \left|\frac{(1-|z|^{2})\lambda}{1-|\lambda|^{2}|z|^{2}} - 1\right|\\[6pt]
	& \le \frac{(1-|\lambda|^{2})|z|}{1-|\lambda|^{2}|z|^{2}} + \frac{(1-|\lambda|)(1+|\lambda||z|^{2})}{1-|\lambda|^{2}|z|^{2}}\\[6pt]
	& \le \frac{|1-\lambda|(1+|z|)}{1-|\lambda||z|}
\end{align}
which shows \eqref{ineq:difference-between-omega-and-identity}.
\end{proof}

\begin{lemma}[Landau]
\label{lemma:Landau}
Let $\omega \in \B$ 
with $ \omega(0)=0$ and $\omega'(0)= \sigma \in (0,1)$.
Then 
$\omega$ is univalent in $\mathbb{D}(\rho)$
with $\rho = \rho( \sigma ) := \sigma/(1+ \sqrt{1-\sigma^2})$, where $\mathbb{D}(r) := \{z \in \C : |z| < r\}$.
\end{lemma}

\noindent
For a proof, see for example \cite[Theorem VI.10 in pp.259--261]{Tsuji:1975}.
We notice that $\lim_{\sigma \nearrow 1} \rho( \sigma) = 1$.

Below when simply saying that an evolution family is \textit{continuous}, it means that it is either of jointly continuous or continuous with respect to the left/rightleft parameter.

Let $\{ \omega_{s,t} \}_{(s,t) \in I_+^2}$ be an evolution family and 
$J$ be a subinterval of $I$.
Then it is clear that
the restriction $\{ \omega_{s,t} \}_{(s,t) \in J_+^2}$
is continuous whenever $\{ \omega_{s,t} \}_{(s,t) \in I_+^2}$
is continuous.
It is a simple exercise that the converse also holds as follows (see also Figure \ref{figure 1} below).

\begin{lemma}
\label{Lemma2-4}
Let $a < b < c$.
If
evolution families
$\{ \phi_{s,t} \}_{(s,t) \in [a,b]_+^2}$
and $\{ \psi_{s,t} \}_{(s,t) \in [b,c]_+^2}$ 
are continuous,
then the extended family 
$\{ \omega_{s,t} \}_{(s,t) \in [a,c]_+^2}$
defined by 
\[
  \omega_{s,t} 
  = 
  \begin{cases}
  \phi_{s,t}, \quad & a \leq s \leq t \leq b,  \\
   \psi_{s,t}, \quad & b \leq s \leq t \leq c,  \\
  \psi_{c,t} \circ \phi_{s,c}, \quad & a \leq s \leq b \leq t \leq c   
  \end{cases}
\]
is also continuous.
\end{lemma}
\begin{figure}[h]
\includegraphics[width=370pt]{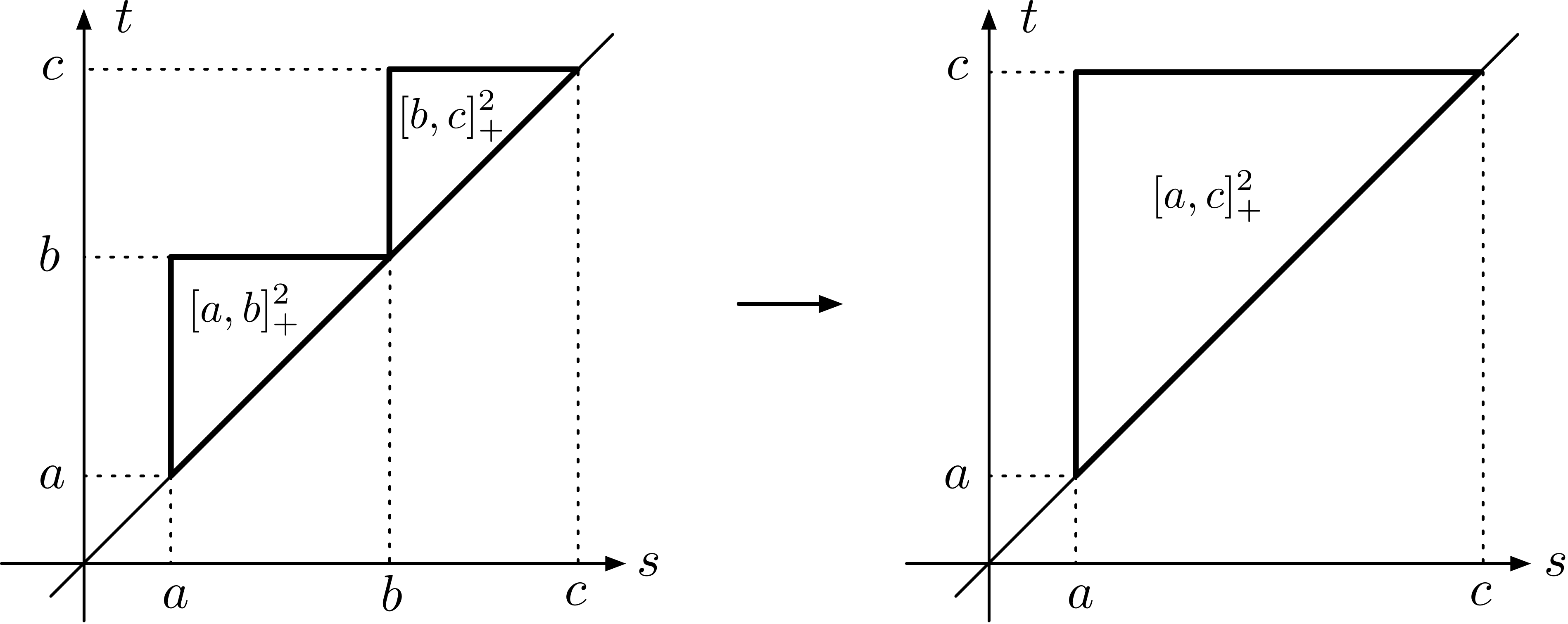}\\[8pt]
\caption{Two continuous evolution families $\{\phi_{s,t}\}$ on $[a,b]_{+}^{2}$ and $\{\psi_{s,t}\}$ on $[b,c]_{+}^{2}$ extend $\{\omega_{s,t}\}$ to be continuous on $[a,c]_{+}^{2}$.}
\label{figure 1}
\end{figure}

%An evolution family $\{ \omega_{s,t} \}_{(s,t) \in I_+^2}$
%is said to fix the origin 
%if $\omega_{s,t}(0) = 0$ for all $(s,t) \in I_+^2$.
\begin{lemma}
\label{lemma:continuity-of-evo-fam-fixing-0}
Let $\{ \omega_{s,t} \}_{(s,t) \in I_+^2}$ with $I = [a,b]$
be an evolution family with $\omega_{s,t}(0) = 0$ for all $(s,t) \in I_+^2$.
Then, 
$\{ \omega_{s,t} \}_{(s,t) \in I_+^2}$ is jointly continuous 
if and only if
$\omega_{a,t}'(0)$ is a continuous function of  $t \in [a,b]$.
In this case 
$\omega_{a,t}'(0) \not= 0$ for all $t \in I$
and $\omega_{s,t}$ is univalent on $\mathbb{D}$
for each $(s,t) \in I_+^2$.
\end{lemma}
\begin{proof}
The necessity is clear.
To show sufficiency let 
$\alpha (t) := \omega_{a,t}'(0)$ for all $t \in I$ and assume 
$\alpha (t)$ is continuous there.
Then by $\omega_{a,t}(z) = \omega_{s,t} ( \omega_{a,s}(z))$,
we have
\begin{equation}
\label{eq:omega_s,t'0}
   \alpha(t) = \omega_{s,t}'(0) \alpha (s) , \quad (s,t) \in I_+^2.
\end{equation}
By the Schwarz lemma,  
$|\alpha (t)|$ is non-increasing in $t$.

We show 
$\alpha (t) \not= 0$ for all $t \in I$.
Suppose, on the contrary, that  
$\alpha (t) = 0$ for some $t \in I$.
Then since
$\alpha (a) = \omega_{a,a}'(0) = 1$,
there exists $t_0 \in (a,b]$ such that 
\begin{align}
& \alpha (t) \not= 0 \quad \text{ for } t \in [a,t_0), \\
& \alpha (t)  = 0 \quad \text{ for }  t \in [t_0,b] .
\end{align}
Take 
$r , \rho \in (0,1)$ arbitrarily
and put $s_0:=a$.
Inductively take $s_n \in (s_{n-1},t_0)$ such that 
$|\omega_{s_{n-1},s_n}'(0)|  = 
\frac{|\alpha (s_n)|}{|\alpha (s_{n-1})|} \leq \rho $ for all $n \in \N$.
Then 
$a=s_0 < s_1< \cdots < s_n < \cdots < t_0$.
For $|z| \leq r $ we have by the Schwarz lemma and  
Lemma \ref{lemma:basic-estimatesII}
\begin{align}
   |\omega_{a,t_0}(z)|
  = & \, |\omega_{s_n,t_0}(\omega_{s_0,s_n}(z))| 
\\
  \leq & \, |\omega_{s_0,s_n}(z)| 
\\
  = & \, |\omega_{s_{n-1},s_n}(\omega_{s_0,s_{n-1}}(z))|  
\\
  \leq & \, |\omega_{s_0,s_{n-1}}(z)| 
 \frac{|\omega_{s_0,s_{n-1}}(z)|+|\omega_{s_0,s_{n-1}}'(0)|}
 {1+ |\omega_{s_0,s_{n-1}}'(0)| |\omega_{s_0,s_{n-1}}(z)| }  
\\
  \leq & \, |\omega_{s_0,s_{n-1}}(z)| 
 \frac{|\omega_{s_0,s_{n-1}}(z)|+\rho}{1+ \rho |\omega_{s_0,s_{n-1}}(z)| }  
\\
  \leq & \, \frac{r+\rho}{1+ \rho r } |\omega_{s_0,s_{n-1}}(z)| 
\\
    \vdots & \, 
\\
 \leq & \, \left( \frac{r+\rho}{1+ \rho r }\right)^n |\omega_{s_0,s_0}(z)|
 = \left( \frac{r+\rho}{1+ \rho r }\right)^n |z| .
\end{align}
Letting 
$n \rightarrow \infty$ 
we have $\omega_{a,t_0}(z) = 0$ for $|z| \leq r$.
Hence by the identity theorem for analytic functions
we obtain $\omega_{a,t_0} = 0$, which contradicts 
the condition (EF1) in Definition \ref{def:evolutiuon-family}.

We show $\{ \omega_{s,t}\}_{(s,t) \in I_+^2}$ is jointly continuous at 
any $(s_0,t_0) \in I_+^2$.
First we consider the case that $s_0=t_0$.
By \eqref{eq:omega_s,t'0} and Lemma \ref{lemma:basic-estimatesII},
we have for $(s,t) \in I_+^2$,
$$
   |\omega_{s,t}(z)-\omega_{t_0,t_0}(z)|
 =
   |\omega_{s,t}(z)-z| 
 \leq
  \left| 1 - \frac{\alpha (t)}{\alpha(s)} \right| \frac{|z|(1+|z|)}{1-|z|} . 
$$
Therefore $\omega_{s,t}(z) \rightarrow \omega_{t_0,t_0}(z)=z$
locally uniformly on $\mathbb{D}$
as $I_+^2 \ni (s,t) \rightarrow (t_0,t_0)$.

Next we consider the case that $s_0 < t_0$.
Here we introduce notations which is used throughout the article; 
$$
x \vee y := \max \{ x , y \}
\quad \textup{and}\quad 
x \wedge y := \min \{ x , y \}
\quad \text{for }  x, y \in \mathbb{R}.
$$
We also use the inequality 
\[
  |\omega(z_1) - \omega (z_0)| 
 \leq \frac{2|z_1-z_0|}{1-r^2}
 \quad \text{ for } \omega \in \B
 \text{ and } |z_0|,|z_1| \leq r,
\]
which is a simple consequence of the Schwarz-Pick inequality.
Then we have 
for $(s,t) \in I_+^2$ with $s<t_0$ and $s_0<t$
and $|z| \leq r$,
\begin{align*}
|\omega_{s,t}(z) &- \omega_{s_0,t_0} (z)|
\\
 \leq & \,
 |\omega_{s,t}(z) - \omega_{s_0,t} (z)|
+ |\omega_{s_0,t}(z) - \omega_{s_0,t_0} (z)|
\\
 = & \,
 |\omega_{s\wedge s_0,t}(z) - \omega_{s \vee s_0,t} (z)|
+ |\omega_{s_0,t\vee t_0}(z) - \omega_{s_0,t \wedge t_0} (z)|
\\
 = &\,
 |\omega_{s \vee s_0, t}(\omega_{s\wedge s_0,s \vee s_0}(z)) - \omega_{s \vee s_0,t} (z)|
+
 |\omega_{t \wedge t_0,t \vee t_0} (\omega_{s_0,t \wedge t_0}(z)) 
 - \omega_{s_0,t \wedge t_0} (z)|
\\[3pt]
  \leq &\,
  \frac{2|\omega_{s\wedge s_0,s \vee s_0}(z) - z|}{1-r^2}
 +
 \left| 1 - \frac{\alpha (t \vee t_0)}{\alpha (t \wedge t_0)} \right|
 \frac{|\omega_{s_0,t \wedge t_0}(z)|(1+|\omega_{s_0,t \wedge t_0}(z)|)}
  {1-|\omega_{s_0,t \wedge t_0}(z)|}
\\[3pt]
 \leq &\,
  \left| 1 - \frac{\alpha (s \vee s_0) }{ \alpha (s \wedge s_0 ) } \right|
  \frac{|z|(1+|z|)}{(1-r^2)(1-|z|)}
 +
 \left| 1 - \frac{\alpha (t \vee t_0)}{\alpha (t \wedge t_0)} \right|
 \frac{|z|(1+|z|)}{(1-|z|)} .
\end{align*}
This implies 
$\omega_{s,t} \rightarrow \omega_{s_0,t_0}$ locally uniformly on $\mathbb{D}$
as $I_+^2 \ni (s,t) \rightarrow (s_0,t_0)$.

Now we show that 
each $\omega_{s_0,t_0}$ is univalent by slightly generalizing 
the argument in \cite{Pom:1965}.
We may assume $s_0< t_0$, since otherwise the univalence is trivial. 
For any $r \in (0,1)$ take $\sigma \in (0,1)$ with 
$\sigma/(1+ \sqrt{1-\sigma^2})  > r$.
Since $\alpha (t)$ is continuous 
and $\alpha (t) \not= 0$ on $[a,b]$,
there exists a division
$s_0 < s_1< \cdots < s_n = t_0$ such that
$\frac{|\alpha (s_k)|}{|\alpha(s_{k-1})|} > \sigma$, $k=1, \ldots , n$.
Then $\omega_{s_{k-1},s_k}$ is univalent in $\D(r)$ by Lemma \ref{lemma:Landau}.
From this and 
$\omega_{s_{k-1},s_k} (\D(r)) \subset \D(r)$
it follows that the composition mapping 
$\omega_{s_0,t_0} = \omega_{s_{n-1},s_n} \circ \cdots \circ \omega_{s_0,s_1} $
is also univalent in $\D(r)$.
Since $r \in (0,1)$ is arbitrary, $\omega_{s_0,t_0}$
is univalent on $\mathbb{D}$.
\end{proof}

\section{Proof of Main Theorems}
\label{Section3}
In this section we will give proofs of the main theorems.
Note that before showing Theorem \ref{thm:continuity-wrt-s} we verify Theorem \ref{thm:cont_on_diagonal_set}, because the equivalence (i)$\Leftrightarrow$(viii) is needed to show Theorem \ref{thm:continuity-wrt-s}.

For 
$\lambda  \in \mathbb{D}$ let
\begin{equation}
   \sigma_{\lambda } (z) := \frac{z+ \lambda }{1 + \overline{\lambda }z}, 
    \quad z \in \mathbb{D} .
\end{equation}
Then $\sigma_{\lambda }$ is an automorphism of $\D$ and we have $\sigma_{\lambda }^{-1} = \sigma_{-\lambda}$

\subsection{Proof of Theorem \ref{thm:right_parameter}}
The implications
(i)$\Rightarrow$(ii)$\Rightarrow$(iii)$\Rightarrow$(iv)
are clear.

Assuming (iv) we show (i) and that 
$\omega_{s,t}$ is univalent on $\mathbb{D}$ 
for each $(s,t) \in I_+^2$.
Let 
\begin{equation}
 c(t) := \omega_{a,t}(z_0), \quad  t \in [a,b].
\end{equation}
Then 
$c(t)$ is a continuous function of $t \in [a,b]$
with $c(a) = \omega_{a,a}(z_0) = z_0$ and satisfies 
\begin{equation}
  \omega_{s,t} (c(s)) 
   = \omega_{s,t} ( \omega_{a,s}(z_0))
   =  \omega_{a,t} (z_0) = c(t) , \quad (s,t) \in I_+^2.
\end{equation}
Set 
\begin{equation}
  \tilde{\omega}_{s,t} 
   := \sigma_{c(t)}^{-1} \circ \omega_{s,t} \circ \sigma_{c(s)} , 
   \quad (s,t) \in I_+^2 .
\end{equation}
Then it is easy to see that 
$\tilde{\omega}_{s,t} (0) = 0$ for $(s,t) \in I_+^2$ and that  
$\tilde{\omega}_{t_1,t_2} \circ \tilde{\omega}_{t_0,t_1} = \tilde{\omega}_{t_0,t_2}$
for $a \leq t_0 \leq t_1 \leq t_2 \leq b$, i.e.,
$\{ \tilde{\omega}_{s,t} \}_{(s,t) \in I_+^2} $ is an evolution family 
fixing the origin.
We also have
\[
   \tilde{\omega}_{s,t}'(0)
   = \sigma_{-c(t)}'(c(t)) \cdot \omega_{s,t}'(c(s)) \cdot \sigma_{c(s)}'(0)
   = \frac{1-|c(s)|^2}{1-|c(t)|^2} \omega_{s,t}'(c(s)) .
\]
Particularly 
\begin{equation}
  \tilde{\omega}_{a,t}'(0)
   = \frac{1-|c(a)|^2}{1-|c(t)|^2} \omega_{a,t}'(c(a))
   = \frac{1-|c(a)|^2}{1-|c(t)|^2} \omega_{a,t}'(z_0) 
\end{equation}
and this implies that $\tilde{\omega}_{a,t}'(0)$ is a continuous 
function of  $t \in  [a,b]$.
Hence by Lemma \ref{lemma:continuity-of-evo-fam-fixing-0}
$\{ \tilde{\omega}_{s,t} \}_{(s,t) \in I_+^2} $ is jointly continuous.
By the continuity 
of $\sigma_{c(t)}$ and $\sigma_{c(s)}^{-1}$,
$\{ \omega_{s,t} \}_{(s,t) \in I_+^2}$ is also jointly continuous.
Similarly by Lemma \ref{lemma:continuity-of-evo-fam-fixing-0}
each $\tilde{\omega}_{s,t} $ is univalent on $\mathbb{D}$
and hence by the injectivity of $\sigma_{c(t)}$ and $\sigma_{c(s)}^{-1}$,
$\omega_{s,t} $ is also univalent on $\mathbb{D}$.
\qed

\subsection{Proof of Theorem \ref{thm:cont_on_diagonal_set}} 
The implication
(i)$\Rightarrow$(vii) is clear.

To show (vii)$\Rightarrow$(viii)
it suffices to see that 
$\{\omega_{s,t} \}_{(s,t) \in I_+^2} $ is hyperbolically bounded.
Take $\varepsilon  \in (0,1)$ arbitrarily. 
Then by (vii), for any $c \in [a,b]$ there exists $\delta_{c} > 0$ such that
\[
  |\omega_{s,t}(0)| = | \omega_{s,t}(0) - \omega_{c,c}(0)| \leq \varepsilon
\]
for all $a \vee (c - \delta_{c}) \leq s \leq t \leq (c + \delta_{c}) \wedge b$.
Consider the open covering 
$$
I \subset \bigcup_{c \in I} (c-\delta_{c},\, c+ \delta_{c}).
$$
Since $I$ is compact, one can choose finite number of open intervals to cover $I$;
$$
I \subset \bigcup_{k=0}^{n} I_{k},\quad \text{where }\,\, I_{k} := (a_{k}, b_{k}) \,\, \text{ and }\,\, a_{k} := c_{k}-\delta_{c_{k}},\,b_{k} := c_{k} + \delta_{c_{k}},\,c_{0} =a,\,c_{n} =b.
$$
By induction, $\{I_{k}\}$ can be a subfamily of intervals such that no interval $I_{k}$ is contained in the union of the others.
With loss of generality we may assume $a_{0}  < a_{1} < \cdots < a_{n-1} < a_{n}$.
Then by a simple consideration, the following inequalities are deduced (see Figure 1);
$$
\begin{cases}
	a_{k+1} < b_{k} \quad \text{ for } k=0, 1, \ldots , n-1,\\
	b_{k} < b_{k+1} \quad \text{ for } k=0, 1, \ldots , n-1,\\
	b_{k} \le a_{k+2} \quad \text{ for } k=0, 1, \ldots , n-2,\\
%	a = c_{0} < c_{1} < \cdots < c_{n} = b.
\end{cases} 
$$
\begin{figure}[h]
\includegraphics[width=400pt]{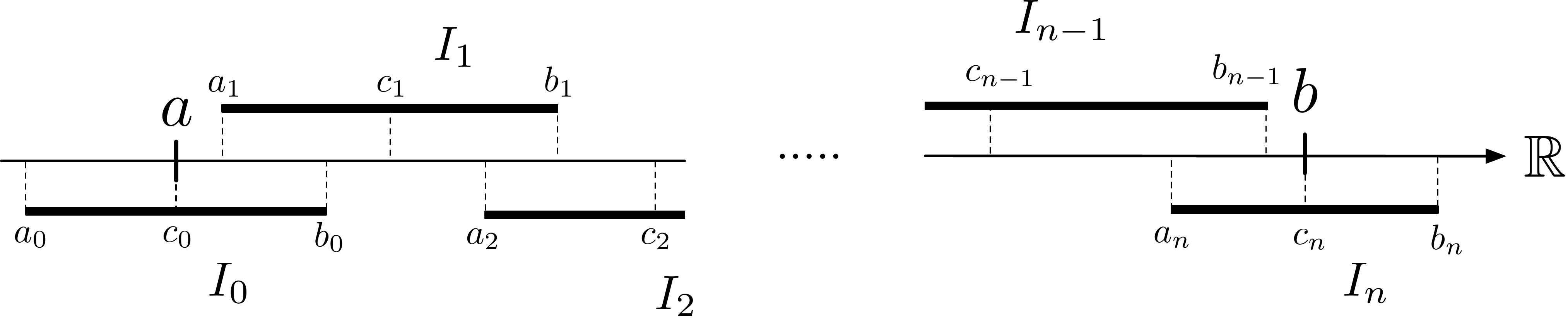}\\[8pt]
\caption{The interval $I = [a,b]$ is covered by finite $I_{k}$'s.}
\end{figure}

\noindent
Set $d_{0} := a,\,d_{n+1} := b$ and take $d_{k}$ arbitrarily as
\[
d_{1} \in (c_{0},b_{0}) \cap I_{1},\quad
d_{k} \in I_{k-1} \cap I_{k} \,(k=2, \cdots, n-1),\quad
d_{n} \in I_{n-1} \cap (a_{n}, c_{n}).
\]
Then $a = d_{0} < d_{1} < \cdots  < d_{n} < d_{n+1} =b$ and $(d_{k}, d_{k+1}) \subset I_{k}$.
Hence
$$
|\omega_{d_{0}, d_{1}}(0)| \leq  \varepsilon, \quad
|\omega_{d_{1}, d_{2}}(0)| \leq  \varepsilon, \quad\cdots,\quad
|\omega_{d_{n}, d_{n+1}}(0)| \leq  \varepsilon.
$$
Set $\varepsilon_{1} := \varepsilon$ and $\varepsilon_{k} := (\varepsilon_{1} + \varepsilon_{k-1})/(1 + \varepsilon_{1}\varepsilon_{k-1}) < 1$.
We then have
\[
  | \omega_{d_{0},d_{2}}(0)| 
  = | \omega_{d_{1}, d_{2}}(\omega_{d_{0}, d_{1}}(0) )| 
 \leq
 \frac{|\omega_{d_{0}, d_{1}}(0)|+ |\omega_{d_{1}, d_{2}}(0)|}
 {1 + |\omega_{d_{1}, d_{2}}(0)||\omega_{d_{0}, d_{1}}(0)|} \leq \varepsilon_2,
\]
Similarly by induction we have
\[
  |\omega_{d_{0}, d_{k}}(0)| = |\omega_{d_{k-1}, d_{k}}(\omega_{d_{0},d_{k-1}}(0))| \leq \varepsilon_{k}, \quad k=2, \ldots , n+1,
\]
Furthermore for any $t \in I$ there exists a unique
$k \in \{0, 1, \ldots, n \}$ such that
$d_{k} \leq t \leq d_{k+1}$.
Then we have 
$$
|\omega_{d_{0},t}(0)| = | \omega_{d_{k}, t}(\omega_{d_{0}, d_{k}}(0) )| 
 \leq
 \frac{|\omega_{d_{k}, t}(0)| + |\omega_{d_{0}, d_{k}}(0)|}
 {1 + |\omega_{d_{k}, t}(0)||\omega_{d_{0}, d_{k}}(0)|} 
  \leq
 \frac{\varepsilon_{0} + \varepsilon_{k-1}}
 {1 + \varepsilon_{0}\varepsilon_{k-1}} 
= \varepsilon_{k} < \varepsilon_{n+1},
$$ 
Since 
$\omega_{s,t} ( \omega_{d_{0},s} (0))= \omega_{d_{0},t}(0) $,
we have by Lemma \ref{lemma:basic-estimatesI}
\[
   |\omega_{s,t}(0)|
 \leq
 \frac{|\omega_{d_{0},s} (0)| + |\omega_{d_{0},t}(0)|}
 {1+|\omega_{d_{0},t}(0)| |\omega_{d_{0},c}(0)|}
 \leq
 \frac{\varepsilon_{n+1}  + \varepsilon_{n+1}}
 {1+ \varepsilon_{n+1}\varepsilon_{n+1}} < 1. 
\]
Therefore $\{ \omega_{s,t}\}_{(s,t) \in I_+^2}$ is hyperbolically bounded.

We show (viii)$\Rightarrow$(i). 
By Theorem \ref{thm:right_parameter} (ii), it suffices to show that 
$\{ \omega_{s,t} \}_{ (s,t) \in I_+^2}$
is continuous with respect to the right parameter.
Let
\[
   \rho := \sup_{(s,t) \in I_+^2} |\omega_{s,t}(0)| \in [0,1) .
\]
Then for any $(s,t) \in I_+^2$
we have
\begin{equation}
\label{eq:global-estimate}
   |\omega_{s,t}(z)| \leq \frac{|z|+\rho}{1+\rho|z|}, \quad 
    z \in \mathbb{D} .
\end{equation}
Notice that by \eqref{ineq:difference-between-omega-and-identity} in Lemma \ref{lemma:basic-estimatesII} and (viii),
$\omega_{s,t} $ converges to $\id_{\mathbb{D}}$ 
locally uniformly on $\mathbb{D}$
for any $c \in [a,b]$ 
as $t-s \searrow 0$ with $s \leq c \leq t$.

For any  $c \in (a,b]$
we show that $\omega_{a,t} \rightarrow \omega_{a,c}$ as 
$t \nearrow c$.
To see this, 
let $\{ t_n \}_{n=1}^\infty$ be a sequence
with $a \leq t_n \nearrow c$ as $n \rightarrow \infty$. 
Since $\{ \omega_{a,t_n} \}_{n=1}^\infty$ is uniformly  bounded, 
there exists a subsequence $\{ \omega_{a,t_{n_k}} \}_{k=1}^\infty$ converging to an analytic function $\omega$ locally uniformly 
on $\mathbb{D}$. 
By \eqref{eq:global-estimate} we have 
$|\omega (z)| \leq \frac{|z|+\rho}{1+\rho |z|}$, $z \in \mathbb{D}$.
Since $\omega_{t_{n_k},c} \rightarrow \id_{\mathbb{D}}$ 
locally uniformly on $\mathbb{D}$, 
we have
\[
  \omega_{a,c}(z) 
 = \lim_{k \rightarrow \infty} \omega_{t_{n_k},c} ( \omega_{a, t_{n_k}}(z)) 
 = \omega (z) .
\]
Thus we have shown that for any sequence 
$\{ \omega_{a,t_n} \}_{n=1}^\infty$ with $t_n \nearrow c$,
there exists a subsequence $\{ \omega_{a,t_{n_k}} \}_{k=1}^\infty$ converging to $\omega_{a,c}$ locally uniformly on $\mathbb{D}$.
Therefore $\omega_{a,t} $ converges 
to $\omega_{a,c}$ locally uniformly on $\mathbb{D}$
as $t \nearrow c$.

Now let $c \in [a,b)$.
Then since $\omega_{c,t} \rightarrow \id_{\mathbb{D}}$ as $t \searrow c$ 
locally uniformly on $\mathbb{D}$,
$\omega_{a,t} = \omega_{c,t} \circ \omega_{a,c}$ 
converges to $\omega_{a,c}$ locally uniformly on $\mathbb{D}$
as $t \searrow c$.
Therefore  $\{ \omega_{a,t} \}_{t \in I}$
is continuous.
\qed

\subsection{Proof of Theorem \ref{thm:continuity-wrt-s}}
The implications (i)$\Rightarrow$(v) and 
(i)$\Rightarrow$(vi) are clear.

Assuming  (vi) we show (i). 
By Theorem \ref{thm:cont_on_diagonal_set} (viii),
it suffices to show $ \2sidelim \omega_{s,t}'(0) = 1$ 
for each $c \in [a,b]$.
Let $a \leq s \leq c \leq t \leq b$.
Then by
$\omega_{s,b}(z) = \omega_{t,b} ( \omega_{s,t}(z))$
we have
\[
   \omega_{s,b}'(0) = \omega_{t,b}'( \omega_{s,t}(0)) \cdot \omega_{s,t}'(0) 
\]
and hence
\[
   \omega_{s,t}'(0)  = \frac{\omega_{s,b}'(0)}{\omega_{t,b}'( \omega_{s,t}(0))} . 
\]
Since $\omega_{s,t}(0) \rightarrow 0$ as $t-s \searrow 0$,
there exist $\delta > 0$ and $\rho_0 \in (0,1)$ such that
$|\omega_{s,t}(0)| \leq \rho_0$ for all $s$, $t$ 
with $s \leq c \leq t $ and $t-s \leq \delta$.
Since $|\omega_{t,b}| \leq 1$,
there exists $M > 0$ such that
$| \omega_{t,b}''(z)| \leq M$ for all $t \in [a,b]$ and $|z| \leq \rho$.
Hence
\[
  |\omega_{t,b}'( \omega_{s,t}(0)) - \omega_{t,b}'(0)| 
 \leq M |\omega_{s,t}(0)| \rightarrow 0
\]
as $t-s \searrow 0$.
Thus
\[
   \2sidelim \omega_{s,t}'(0)
 = \2sidelim \frac{\omega_{s,b}'(0)}{\omega_{t,b}'(0)} 
 = \frac{\omega_{c,b}'(0)}{\omega_{c,b}'(0)} = 1 .
\]

To show that (v)$\Rightarrow$(i), we need the following lemma.

\begin{lemma}
\label{Lemma3-1}
Suppose $\{ \omega_{s,t}\}_{(s,t) \in [a , b]_+^2}$ 
is an evolution family which is  continuous with respect to the left parameter 
and hyperbolically bounded with 
$\rho = \sup_{a \leq s \leq t \leq b} |\omega_{s,t}(0)|$.
If there exists $\gamma \in (a, b )$ and $r \in (0,1)$ such that 
$\omega_{\gamma , b}$ is univalent in 
$\mathbb{D}((r + \rho)/(1+\rho r ))$,
then
$\{ \omega_{s,t}\}_{(s,t) \in [a , b]_+^2}$ is jointly continuous
at any $(s_0, \gamma)$ with $a \leq s_0 \leq \gamma$.
\end{lemma}
\begin{proof}
It suffices to show that 
for any sequence $\{ (s_n, t_n ) \}_{n=1}^\infty$ in  $[a , b]_+^2$
with $(s_n , t_n ) \rightarrow (s_0, \gamma ) $,
one can choose a subsequence  $\{ (s_{n_k}, t_{n_k} ) \}_{k=1}^\infty$
such that $\omega_{s_{n_k}, t_{n_k}} \rightarrow \omega_{s_0, \gamma }$
locally uniformly in $\mathbb{D}$.

Since $\{ \omega_{s_n, t_n } \}_{n=1}^\infty$ is bounded and forms a normal family,
there is a subsequence $\{ \omega_{s_{n_k}, t_{n_k} } \}_{k=1}^\infty$
which converges to a function $\omega \in \mathcal{A} $ 
locally uniformly in $\mathbb{D}$.
Notice that by Lemma 2.1,
\[
   |\omega (z)|
   = \lim_{k \rightarrow \infty} |\omega_{s_{n_k}, t_{n_k}}(z)|
   \leq
   \lim_{k \rightarrow \infty} 
   \frac{|z| + |\omega_{s_{n_k}, t_{n_k}}(0) |}{1+ |\omega_{s_{n_k}, t_{n_k}}(0) ||z| } 
   \leq
    \frac{|z|+ \rho}{1+\rho |z|} .
\]
From the identity
\[
   \omega_{t_{n_k}, b } ( \omega_{s_{n_k}, t_{n_k}}(z))
 =
 \omega_{s_{n_k}, b} (z)
\]
it follows by letting $k \rightarrow \infty$ that 
\[
   \omega_{\gamma , b} (\omega(z)) = \omega_{s_0 , b } (z) 
 = \omega_{\gamma , b} (\omega_{s_0, \gamma }(z)) .
\]
Since both $\omega (z)$ and $\omega_{s_0,t_0} (z)$ belong to $\mathbb{D}((r+\rho)/(1+r \rho ))$ for all $z$ with $|z|< r$ and $\omega_{\gamma , b}$ is univalent in $\mathbb{D}((r+\rho)/(1+r \rho ))$,
we obtain $\omega (z) = \omega_{s_0, \gamma }(z)$ for $|z|< r$.
By the identity theorem for analytic functions 
we have $\omega = \omega_{s_0, \gamma }$,  and hence 
$\omega_{s_{n_k}, t_{n_k}} \rightarrow \omega_{s_0, \gamma }$
locally uniformly in $\mathbb{D}$.
\end{proof}

We show (v)$\Rightarrow$(i). Let $\rho := \sup_{a \leq s \leq t \leq b} |\omega_{s,t}(0)|$
and 
\begin{align*}
   J := & \, \left\{ c \in [a,b] : 
  \{ \omega_{s,t} \}_{[a,b]_+^2 } \text{ is jointly continuous on } [c,b]_+^2 
 \text{ and the} \right.
\\[-5pt]
   &   \quad \left. \text{ subfamily }\{ \omega_{s,t} \}_{[a,c]_+^2 } 
 \text{ is continuous with respect to the left parameter} \right\}.
\end{align*}
Since $b \in J$ and $J \not=\emptyset$,  
it suffices to show that $J$ is open and closed in $[a,b]$.

{\bf 1}. 
We show that  if $c \in J$, then $[c,b] \subset J$. 
Let $c < c' \leq b$.
Then it is obvious that 
$\{ \omega_{s,t}\}$ is jointly continuous on $[c',b]_+^2$,
and that the mapping 
$[a,c] \ni s \mapsto \omega_{s,c'} = \omega_{c,c'} \circ \omega_{s,c} \in \mathcal{A}$ 
is continuous.
Because $\{ \omega_{s,t} \}$ is jointly continuous on $[c,b]_+^2$,
the mapping 
$[c, c'] \ni s \mapsto \omega_{s,c'}  \in \mathcal{A}$ 
is also continuous.
Therefore the extended mapping 
$[a, c'] \ni s \mapsto \omega_{s,c'}  \in \mathcal{A}$ is continuous,
i.e.,
$\{ \omega_{s,t} \}_{[a,c']_+^2 } $  is continuous with respect to the left parameter.
Thus $c' \in J$.

{\bf 2}.
We show that $J$ is open.
Let $c \in J$ and $r$ be fixed with $0<r<1$.
If $c=a$, then by {\bf 1}, $c$ is an interior point of the interval $[a,b]$.
Assume $a < c \leq b$.
Take $\sigma $ as
$0< \sigma < 1$ and 
\[
   \frac{r+\rho}{1+\rho r} < \frac{\sigma}{1+\sqrt{1-\sigma^2}} .
\]
Since $\omega_{s,c}(0) \rightarrow \omega_{c,c}(0)=0$
and $\omega_{s,c}'(0) \rightarrow \omega_{c,c}'(0)=1$,
there exists $\delta > 0$ such that
\[
   \frac{|\omega_{s,c}'(0)}{1-|\omega_{s,c}(0)|^2} > \sigma \quad
  \text{ for } \,c-\delta \leq s \leq c .
\]
Let 
\[
  \tilde{\omega}_{s,c}(z)
 = \frac{\omega_{s,c}(z)- \omega_{s,c}(0)}
 {1- \overline{\omega_{s,c}(0)} \omega_{s,c}(z)}, \quad z \in \mathbb{D} .
\]
Then
$\tilde{\omega}_{s,c} \in \mathfrak{B}$,
$\tilde{\omega}_{s,c}(0) = 0$ and 
\[
   |\tilde{\omega}_{s,c}'(0)|
  = \frac{|\omega_{s,c}'(0)}{1-|\omega_{s,c}(0)|^2} > \sigma \quad
 \text{ for } \,c-\delta \leq s \leq c .
\]
Therefore by Lemma 2.?, for $c-\delta \leq s \leq c$,  
$\tilde{\omega}_{s,c}$ is univalent 
in $\mathbb{D}(\sigma/(1+\sqrt{1-\sigma^2}))$
and so is $\omega_{s,c}$.
Since 
$\mathbb{D}( (r+\rho)/(1+\rho r)) \subset \mathbb{D}(\sigma/(1+\sqrt{1-\sigma^2}))$,
by applying Lemma \ref{Lemma3-1} we have that
$\{ \omega_{s,t} \}_{[a,b]_+^2}$ is jointly continuous on the set (see Figure \ref{figure3})
\[
     A_{c-\delta , c} : = \bigcup_{c - \delta \leq t \leq c } \{ (s,t) : a \leq s \leq t \} .
\]
Specially this implies for each $t$ with $c-\delta \leq t \leq c$
the mapping $[a,t] \ni s \mapsto \omega_{s,t} \in \mathcal{A}$ is continuous, i.e.,
$\{ \omega_{s,t} \}_{[a,t]_+^2}$ is continuous with respect to the left parameter.

\begin{figure}[h]
\includegraphics[width=380pt]{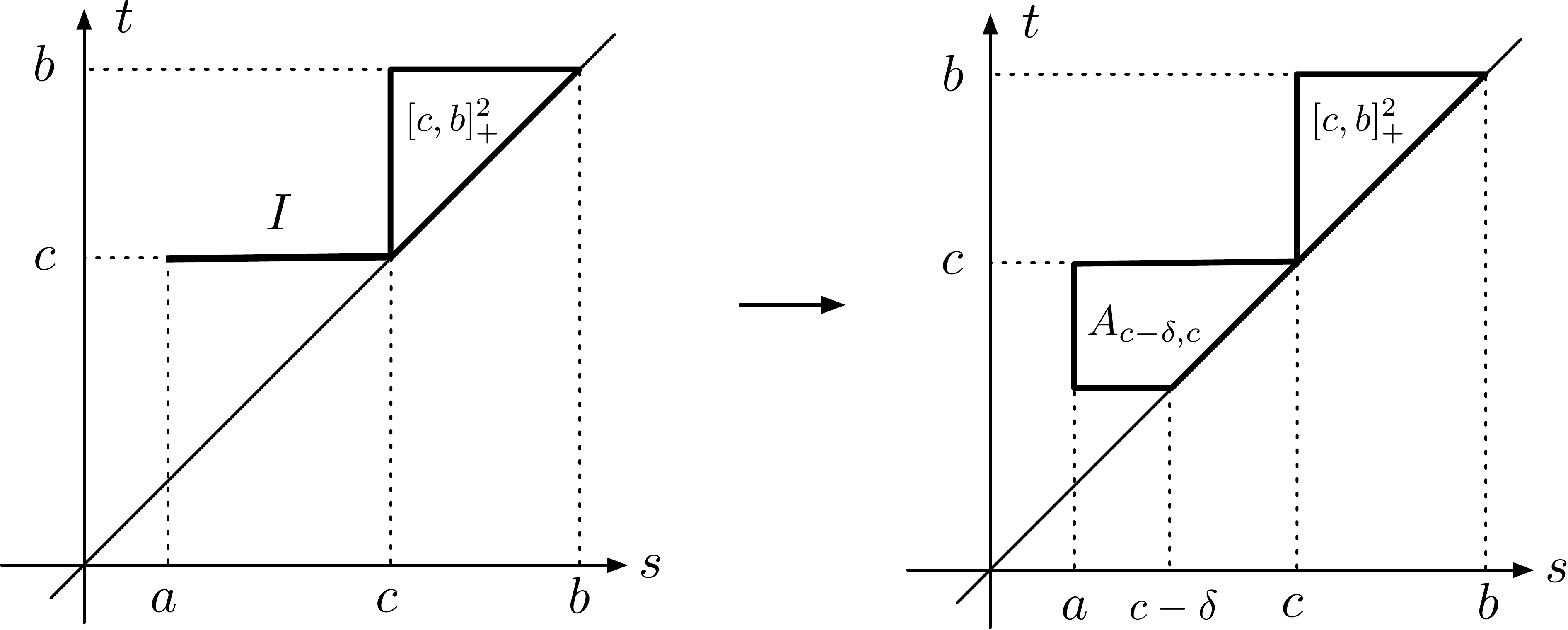}\\[8pt]
\caption{Left: If $c \in J$, then $\{ \omega_{s,t} \}$ is jointly continuous on $[c,b]_{+}^{2}$ and continuous w.r.t. the left parameter on $I$. 
Right: Then there exists $\delta >0$ such that $\{ \omega_{s,t} \}$ is jointly continuous on $A_{c-\delta,c}$. Applying Lemma \ref{Lemma2-4}, $\{ \omega_{s,t} \}$ is jointly continuous on $[c-\delta, b]_{+}^{2}$.}
\label{figure3}
\end{figure}

Furthermore $\{ \omega_{s,t} \}_{[a,b]_+^2}$ is jointly continuous
on $[c-\delta, c]_+^2$. 
Combining this and the joint continuity 
on $[c, b]_+^2$ we obtain  that 
$\{ \omega_{s,t} \}_{[a,b]_+^2}$ is jointly continuous 
on $[c-\delta, b]_+^2$ by Lemma \ref{Lemma2-4}.
Thus $[c-\delta , c] \subset J$.
By {\bf 1} we also have $[c,b] \subset J$.
Hence $[c-\delta, b] \subset J$ and we conclude $J$ is open.

{\bf 3}.
We show that $J$ is closed.
To see this let $J \ni c_n \searrow c_0$.
Then, since $\{ \omega_{s,t} \}$ is jointly continuous on $[c_n,b]_+^2$ 
for all $n \in \mathbb{N}$,
$\omega_{s , b}$ is univalent in $\mathbb{D}$ for each $s$ with $c_0 < s \leq b$.
Thus by Hurwitz's theorem and (EF1), 
the limit $\omega_{c_0,b} = \lim_{n \rightarrow \infty} \omega_{c_n, b}$ 
is also univalent in $\mathbb{D}$.
It follows from Lemma that 
$\{ \omega_{s,t} \}$ is jointly continuous on $A_{c_0 , b} = \bigcup_{c_0 \leq t \leq b} \{ (s,t) : a \leq s \leq t \} .$ 
By repeating the argument in {\bf 2} we conclude $c_0 \in J$.
\qed

\section{Remarks}
\label{Section4}
\subsection{Remarks on the main theorems}
Comparing with Theorem \ref{thm:right_parameter} and Theorem \ref{thm:continuity-wrt-s}, 
the continuity in the right parameter seems to be a stronger assumption than left one.
One of the reasons for this can be seen in terms of the theory of Loewner chains as follows.

For an evolution family $\{\omega_{s,t}\}_{0 \le s \le t < \infty}$ and a fixed $T > 0$, $\{f_{t} := \omega_{t,T}\}_{t \in [0,T]}$ is called a \textit{Loewner chain}.
If $\{\omega_{s,t}\}$ is jointly continuous, then by Theorem \ref{thm:right_parameter} both $\{\omega_{s,t}\}$ and $\{f_{t}\}$ are families of univalent functions on $\D$ and the following is satisfied;
\begin{equation}
\label{evol-Loewner}
\omega_{s,t} = f_{t}^{-1} \circ f_{s},
\end{equation}
that is an easy consequence of the equation $\omega_{s,T} = \omega_{t,T} \circ \omega_{s,t}$.
Equation \eqref{evol-Loewner} implicitly tells us that the continuity in the right parameter $t$ will deduce that $f_{t}^{-1}$ is well-defined (namely $f_{t}$ is univalent on $\D$) and continuous in $t$ (namely so is $f_{s}$ in $s$ as well).
In contrast, if $\{\omega_{s,t}\}$ is continuous with respect to the left parameter $s$, this implies that $f_{s}$ is well-behaved but nothing says about $f_{t}^{-1}$.

In view of Theorem \ref{thm:continuity-wrt-s}, the following open problem arises.

\begin{problem}
If an evolution family $\{\omega_{s,t}\}$ is continuous with respect to the left parameter, then is it jointly continuous?
\end{problem}
The answer of this problem is probably negative, though we do not have any idea to show it.
If there exists a counterexample, it should be non-hyperbolically bounded.

\subsection{A discontinuous evolution family}

Following the arguments in the proofs of our main theorems, 
one finds that the semigroup property (EF3) in Definition \ref{def:evolutiuon-family} works quite effectively at key points of the proofs. Hence the reader may expect that even some continuity in a parameter is deduced from the assumption (EF1)-(EF3). Here we give, however, an example of discontinuous evolution families. 

Let $f : \R \to \R$ be an \textit{additive function}, i.e., $f$ satisfies
\begin{equation}
\label{additive-function}
f(x + y) = f(x) + f(y) \quad \text{ for all $x, y \in [0,\infty)$}.
\end{equation}
It is known that there exists a discontinuous additive function, sometimes called the \textit{Hamel function}.
For details of additive functions and the Hamel functions, see \cite[pp.128--130]{MR2467621}. 
Remark that it follows from \eqref{additive-function} that $f(0)=0$.

By using the Hamel function $f$, we define $\omega_{s,t}$ as
\begin{equation}
\omega_{s,t} := e^{i f(t-s)}z \quad (z \in \D,\,(s,t) \in [0,\infty)_{+}^{2}).
\end{equation}
Then it is easy to see that the family $\{\omega_{s,t}\}_{[0,\infty)_{+}^{2}}$ is an evolution family, but not continuous with respect to both left and right parameters.

\bibliographystyle{amsalpha}
%\bibliographystyle{amsplain}
%\bibliography{/Users/ikkeihotta/Dropbox/bibdata.bib}

\begin{thebibliography}{CDMG14}

\bibitem[BCDM12]{BracciCD:evolutionI}
F.~Bracci, M.~D. Contreras, and S.~D{\'{\i}}az-Madrigal, \emph{Evolution
  families and the {L}oewner equation {I}. {T}he unit disc}, J. Reine Angew.
  Math. \textbf{672} (2012), 1--37.

\bibitem[CDMG14]{contreraslocalduality}
M.~D. Contreras, S.~D{\'{\i}}az-Madrigal, and P.~Gumenyuk, \emph{Local duality
  in {L}oewner equations}, J. Nonlinear Convex Anal. \textbf{15} (2014), no.~2,
  269--297.

\bibitem[FHS20]{MR4152669}
U.~Franz, T.~Hasebe, and S.~Schlei{\ss}inger, \emph{Monotone increment
  processes, classical {M}arkov processes, and {L}oewner chains},
  Dissertationes Math. \textbf{552} (2020), 119.

\bibitem[Hei62]{Heins}
M.~Heins, \emph{Selected topics in the classical theory of functions of a
  complex variable}, Athena Series: Selected Topics in Mathematics, Holt,
  Rinehart and Winston, New York, 1962.

\bibitem[HH21]{HasebeHotta2021}
T.~Hasebe and I.~Hotta, \emph{Additive processes on the unit circle and
  {L}oewner chains}, Int. Math. Res. Not. IMRN (2021), to appear.

\bibitem[Kuc09]{MR2467621}
M.~Kuczma, \emph{An introduction to the theory of functional equations and
  inequalities}, second ed., Birkh\"{a}user Verlag, Basel, 2009.

\bibitem[Pom65]{Pom:1965}
Ch. Pommerenke, \emph{\"{U}ber die {S}ubordination analytischer {F}unktionen},
  J. Reine Angew. Math. \textbf{218} (1965), 159--173.
  
\bibitem[Tsu75]{Tsuji:1975}
M.~Tsuji, \emph{Potential theory in modern function theory}, Chelsea Publishing
  Co., New York, 1975.

\bibitem[Yan]{Yanagihara:2022}
H.~Yanagihara, \emph{Lowener theory on analytic universal covering maps},
  arXiv:1907.11987.

\end{thebibliography}

\def\cprime{$'$}
\providecommand{\bysame}{\leavevmode\hbox to3em{\hrulefill}\thinspace}
\providecommand{\MR}{\relax\ifhmode\unskip\space\fi MR }
% \MRhref is called by the amsart/book/proc definition of \MR.
\providecommand{\MRhref}[2]{%
  \href{http://www.ams.org/mathscinet-getitem?mr=#1}{#2}
}
\providecommand{\href}[2]{#2}

\end{document}